\documentclass[12pt,reqno]{amsart} 

\usepackage{times}

\usepackage[
  margin=1in
]{geometry}

\usepackage{amssymb,amsfonts,amsthm}
\usepackage{marginnote}
\usepackage{comment}
\usepackage{enumitem}
\usepackage[dvipsnames]{color}

\usepackage[colorlinks=true, pdfstartview=FitV, linkcolor=black, citecolor=blue, urlcolor=blue]{hyperref}

\def\Xint#1{\mathchoice
{\XXint\displaystyle\textstyle{#1}}%
{\XXint\textstyle\scriptstyle{#1}}%
{\XXint\scriptstyle\scriptscriptstyle{#1}}%
{\XXint\scriptscriptstyle%
\scriptscriptstyle{#1}}%
\!\int}
\def\XXint#1#2#3{{\setbox0=\hbox{$#1{#2#3}{%
\int}$ }
\vcenter{\hbox{$#2#3$ }}\kern-.6\wd0}}
\def\barint{\, \Xint -} 
\def\bariint{\barint_{} \kern-.4em \barint}
\def\bariiint{\bariint_{} \kern-.4em \barint}
\renewcommand{\iint}{\int_{}\kern-.34em \int} 
\renewcommand{\iiint}{\iint_{}\kern-.34em \int} 

\DeclareMathAlphabet{\mathcal}{OMS}{cmsy}{m}{n}

\theoremstyle{plain}
\newtheorem{theorema}{Theorem}

\newtheorem{theoremabis}{Theorem}

\newtheorem{theorem}{Theorem}[section]

\newtheorem{definition}[theorem]{Definition}
\newtheorem{lemma}[theorem]{Lemma}

\theoremstyle{definition}
\newtheorem{remark}[theorem]{Remark}

\newcommand{\R}{\mathbb{R}}

\newcommand{\p}{\partial}

\newcommand{\ep}{\varepsilon}

\let\tilde\relas
\newcommand{\tilde}[1]{\widetilde{#1}}

\definecolor{darkgreen}{rgb}{0,0.5,0}
\definecolor{darkblue}{rgb}{0,0,0.7}
\definecolor{darkred}{rgb}{0.9,0.1,0.1}
\definecolor{lightblue}{rgb}{0,0.51,1}

\numberwithin{equation}{section}
\setlist[enumerate]{leftmargin=*}

\title{Epsilon regularity for the Navier-Stokes equations via 
weak-strong uniqueness}
\author[D. Albritton]{Dallas Albritton}
\address[D. Albritton]{Department of Mathematics,
Princeton University,
Princeton, NJ 08544, USA}
\email{dalbrit@princeton.edu}

\author[T. Barker]{Tobias Barker}
\address[T. Barker]{Department of Mathematical Sciences, University of Bath, Bath BA2 7AY, UK}
\email{tobiasbarker5@gmail.com}

\author[C. Prange]{Christophe Prange}
\address[C. Prange]{Cergy Paris Universit\'e, Laboratoire de Math\'ematiques AGM, UMR CNRS 8088, France}
\email{christophe.prange@cyu.fr}

\dedicatory{In honor of Olga Lady\v{z}enskaja}

\date{\today}

\begin{document}
\begin{abstract}
We give a new concise proof of a certain one-scale epsilon regularity criterion using weak-strong uniqueness for solutions of the Navier-Stokes equations with non-zero boundary conditions. It is inspired by an analogous approach for the stationary system due to Struwe.

\end{abstract}

\maketitle

\section{Introduction}
In the regularity theory of the three-dimensional non-stationary Navier-Stokes equations, so-called epsilon regularity theory remains the state-of-the-art. As well as being of independent interest, epsilon regularity at one scale (`one-scale epsilon regularity') has been a crucial tool for proving some of the best results in the field. See \cite{CKN}, \cite{Necas}, \cite{tsai} and \cite{ESS}, for example. Thus far, there have been three main methods for proving one-scale epsilon regularity results:
\begin{itemize}
\item direct iteration arguments (see \cite{CKN}),
\item compactness arguments (see \cite{lin}, \cite{ladyzhenskayaseregin}),
\item De Giorgi techniques (see \cite{vasseur}).
\end{itemize}
This short article is devoted to a new concise proof of a certain one-scale epsilon regularity criterion for the three-dimensional non-stationary Navier-Stokes equations away from boundaries. Theorem~\ref{theom.epreg} is based on slicing techniques and a comparison to a solution 
for which we have improved integrability. To the best of our knowledge, related arguments in the Navier-Stokes context were 
introduced by Struwe in \cite{struwe1988partial} for the five-dimensional stationary Navier-Stokes equations. Our paper is, in spirit, an analogue for the non-stationary 
Navier-Stokes equations. 

The following localized weak-strong uniqueness result is the keystone of our paper.

\begin{theorema}[first version of weak-strong uniqueness with boundary conditions]\label{theom.locWSU}
Let $\Omega\subset\R^3$ be a smooth bounded domain\footnote{We rely on linear results of \cite{FKS11}, which require $\partial\Omega$ to be of class $C^{2,1}$.\label{foot-smoothbdary}} and $T>0$. 
There exists $\bar\kappa=\bar\kappa_{\Omega,T}\in(0,\infty)$ and $C(\Omega,T)\in (0,\infty)$ such that the following holds.\footnote{See \eqref{e.condbarkappa} for an estimate of $\bar\kappa_{\Omega,T}$. One can take $C(\Omega,T):= K_{\Omega,T}(1+C_0C_2)$. Here $C_0$ is defined in \eqref{bilinearest}, $C_{2}$ is defined in \eqref{C2def} and $K_{\Omega,T}\in(0,\infty)$ is defined in \eqref{e.linestFLR}.} There exists a unique very weak solution 
\begin{equation}\label{e.intL4U}
U\in L^4(\Omega\times (0,T))
\end{equation}
to the Navier-Stokes equations on $\Omega\times(0,T)$ in the sense of Definition \ref{def.vws} with boundary data $a$ and initial data $b$ satisfying the integrability conditions
\begin{equation}\label{e.intassab}
a\in L^{4}(\partial\Omega\times (0,T)),\qquad b\in L^{4}(\Omega),
\end{equation}
the compatibility conditions
\begin{equation}\label{e.struc}
\int\limits_{\partial\Omega}a(\cdot,t)\cdot n=0,\qquad \int\limits_\Omega b\cdot\nabla q=0 \quad \mbox{for all}\quad q\in C^\infty(\R^3)
\end{equation}
and the smallness condition
\begin{equation}\label{e.smallnessab}
\kappa:=\|a\|_{L^{4}(\partial\Omega\times(0,T))}+\|b\|_{L^4(\Omega)}<\bar\kappa.
\end{equation}

Moreover, $U\in L^4(0,T;L^6(\Omega))$,
and
\begin{equation}\label{e.estUL4L6}
\|U\|_{L^4(0,T;L^6(\Omega))}\leq C(\Omega,T)\kappa.
\end{equation}
\end{theorema}
The critical integrability in~\eqref{e.estUL4L6} is the main practical outcome of the theorem.

Below, we give a second version of this localized uniqueness result, see Theorem \ref{theom.locWSUbis}, that can be directly used to prove the epsilon regularity result stated in Theorem \ref{theom.epreg} below.

\begin{remark}[structure result]\label{rem.struc}
We actually prove a strengthened version of the uniqueness result. Namely, we prove that any very weak solution $U$ in the sense of Definition \ref{def.vws} with data satisfying \eqref{e.intassab}, \eqref{e.struc} and the smallness condition \eqref{e.smallnessab} can be written as $U=\bar U+V$, where $\bar U$ is the unique very weak solution to the Stokes equation with boundary data $a$ and initial data $b$ and $V$ is the unique mild solution to the perturbed Navier-Stokes equations around $\bar U$ in $\Omega\times (0,T)$ with homogeneous boundary and initial data. We refer to Section \ref{sec.wsu} below for more details.
\end{remark}

\begin{remark}[smallness assumption \eqref{e.smallnessab}]
The smallness condition~\eqref{e.smallnessab} is required in Section~\ref{sec.wsu} (Step~1-b) below in order to construct a mild solution to the perturbed Navier-Stokes equations around the solution $\bar U$ to the linear Stokes equations lifting the boundary data~$a$ and the initial data~$b$.
\end{remark}
\begin{remark}[qualitative integrability assumption $L^{4}$]\label{rem.L4+}
The integrability condition \eqref{e.intL4U} on the solution is essential to prove that the perturbation in Section \ref{sec.wsu} (Step 2) has finite energy. The integrability condition \eqref{e.intassab} on the boundary data $a$ is assumed in order to apply the linear results of Farwig, Kozono and Sohr \cite{FKS11} (see Section \ref{sec.prelim} below). Notice that according to \cite{FKS11}, the boundary data can be taken in the larger space $L^4(0,T;W^{-\frac16,6}(\partial\Omega))$. However, we do not state such an improved result, because in the application to epsilon regularity that we have in mind, the data $a$ satisfies \eqref{e.intassab}.

\end{remark}

In the second version of the localized uniqueness stated below, the main change with respect to the previous theorem is in the conditions on the initial data $b$. Indeed the requirement \eqref{e.struc} is strong and imposes not only that $b$ is divergence-free, but also that $b\cdot n=0$ on $\partial\Omega$. This condition cannot be satisfied in general for data $b$ arising from slicing in the proof of the epsilon regularity result, Theorem \ref{theom.epreg} below; for details see Step 1 in Section \ref{sec.proof}. The assumptions in Theorem \ref{theom.locWSUbis}, contrary to those of Theorem \ref{theom.locWSU}, are immediately satisfied for the data stemming from the slicing in the proof of Theorem \ref{theom.epreg}. In order to fit into the framework of Theorem \ref{theom.epreg}, we take $\Omega=B_{r_0}$ for some $r_0\in (\frac58,\frac78)$. This specific choice is only for convenience. It is easy to extend the result to more general domains, the key point being that one has to assume that $b$ is defined and divergence-free on a larger domain than $\Omega$, so that it is possible to cut-off.

\begin{theoremabis}[second version of weak-strong uniqueness with boundary conditions]\label{theom.locWSUbis}
Let $\Omega=B_{r_0}$ for some $r_0\in (\frac58,\frac78)$ and $t_0\in(-1,-\frac34)$. There exists universal constants $\bar\kappa\in(0,\infty)$ and $C\in (0,\infty)$ such that the following holds.\footnote{See \eqref{e.condbarkappabis} for an estimate of $\bar\kappa$. One can take $C=K(1+C_0C_2$). Here $C_0$ is defined in \eqref{bilinearest}, $C_{2}$ is defined in \eqref{C2def} and $K\in (0,\infty)$ is a universal constant (see \eqref{e.linestFLRbis}).} There exists a unique very weak solution 
\begin{equation}\label{e.intL4U2.0}
U\in L^4(B_{r_0}\times(t_0,0))
\end{equation}
to the Navier-Stokes equations on $B_{r_0}\times(t_0,0)$ in the sense of Definition~\ref{def.vws} with boundary data $a$ and initial data $b$ satisfying the integrability conditions
\begin{equation}\label{e.intassab2.0}
a\in L^{4}(\partial B_{r_0}\times(t_0,0)),\qquad b\in L^{4}(B_1),\footnote{Notice here that $b$ is defined on a domain strictly larger than $B_{r_0}$ so that one can cut-off $b$.}
\end{equation}
the compatibility condition
\begin{equation}\label{e.strucbis}
\int\limits_{\partial B_{r_0}}a(\cdot,t)\cdot n=0,
\end{equation}
the incompressibility condition $\nabla\cdot b=0$ in $B_1$ in the sense of distributions 
and the smallness condition
\begin{equation}\label{e.smallnessab2.0}
\kappa:=\|a\|_{L^{4}(\partial B_{r_0}\times(t_0,0))}+\|b\|_{L^4(B_1)}<\bar\kappa.
\end{equation}

Moreover, $U\in L^4(t_0,0;L^6(B_{r_0}))$,
and
\begin{equation}\label{e.estUL4L62.0}
\|U\|_{L^4(t_0,0;L^6(B_{r_0}))}\leq C\kappa.
\end{equation}
\end{theoremabis}
Again, the critical integrability is the main practical outcome of the theorem.

We now state the epsilon regularity result. Let $(U,P)$ be a finite-energy weak solution to the three-dimensional Navier-Stokes equations in $Q_1=B_1(0)\times(-1,0)$ i.e. 
\begin{align}
&\partial_tU-\Delta U+U\cdot\nabla U+\nabla P=0,\quad\nabla\cdot U=0\qquad \mbox{in}\quad Q_1,\nonumber\\
&\qquad\qquad\mbox{in the sense of distributions, and}\label{e.distribnse}\\
&\Bigg(\sup_{t\in(-1,0)}\int\limits_{B_1}|U(\cdot,t)|^2+\int\limits_{-1}^0\int\limits_{B_1}|\nabla U|^2\Bigg)^{\frac{1}{2}}\leq M
<\infty.\label{e.en}
\end{align}

\begin{theorema}[epsilon regularity]\label{theom.epreg} 
There exists $\bar\ep\in(0,1)$  
such that for any $M\in(0,\infty)$, for any finite-energy weak solution $U$ to the Navier-Stokes equations in the sense of \eqref{e.distribnse}-\eqref{e.en} belonging to $C^\infty(B_1\times(-1,T))$ for all $T\in(-1,0)$,\footnote{\label{foot.fts}This assumption is there to keep technicalities to a minimum. It makes our result applicable to rule out first-time singularities.} 
the following result holds. 
\begin{description}
\item[Qualitative statement] Assume that $U$ satisfies the smallness condition in Theorem \ref{theom.locWSUbis}, i.e. 
\begin{equation}\label{e.smallnessUbis}
\|U\|_{L^{4}(Q_1)}<\bar\ep.
\end{equation}
Then $U\in L^\infty(Q_{\frac14})$.
\item[Quantitative statement] Let $0<\ep<\bar\ep$. Assume that
\begin{equation}\label{e.HL4+}
\|U\|_{L^{4}(Q_1)}\leq\ep.
\end{equation}
Then $U\in L^\infty(Q_{\frac14})$ and in addition we have the quantitative estimate
\begin{equation}\label{e.concl}
\|U\|_{L^\infty(Q_{\frac14})}\leq P(\ep,M),
\end{equation}
where $P(\ep,M)$ is a positive polynomial in $0<\ep,\, M$.

\end{description}
\end{theorema}

This result is a mere corollary of Theorem \ref{theom.locWSUbis}.

\begin{remark}[smallness condition~\eqref{e.smallnessUbis}]
Ultimately, the reason for the space $L^4(Q_1)$ in Theorem~\ref{theom.epreg} is that the solution of the linear Stokes equation with boundary data in $L^{4}(\partial\Omega\times (0,T))$ and zero initial condition belongs to the critical space $L^{4}(0,T; L^{6}(\Omega))$ (see Theorem~\ref{theo.flrex}). Notably, the proof we present below apparently does not work in $L^{4-\varepsilon}(Q_1)$.
\end{remark}

\begin{remark}[nonlocality and pressure]
Theorem~\ref{theom.epreg} is in the spirit of some other epsilon regularity results that do not involve the pressure, such as \cite{wolf15} and \cite{kwon2021role}. 

\end{remark}

\subsection*{Outline of the paper}

In Section~\ref{sec.prelim}, we review the concept of very weak solution with $L^p$ initial and boundary data. In Section~\ref{sec.wsu}, we prove weak-strong uniqueness, Theorems~\ref{theom.locWSU} and \ref{theom.locWSUbis}. In Section~\ref{sec.proof}, we prove the epsilon regularity criterion, Theorem \ref{theom.epreg}.

\subsection*{Notations}

For $r>0$ and $p\in(1,\infty)$, we denote by $\mathbf A$ the Stokes operator realized on $L^p_\sigma(B_r)$. Notice that $\mathbf A=-\mathbb P\Delta_D$, where $\mathbb P$ is the Helmholtz-Leray projection and $\Delta_D$ is the realization of the Laplace operator under the Dirichlet boundary condition on $\partial B_r$. 
We also use the notation $(e^{-t\mathbf A})_{t\in(0,\infty)}$ for the Stokes semigroup on $B_{r}$. Notice that these notations do not involve explicitly the parameter $r$ in order to lighten the notation and because the parameter $r$ will be fixed. For further details concerning the Stokes operator and the Stokes semigroup, we refer to \cite[Chapter III.2 and IV.1]{SohrNSEbook} and \cite{Giga86}. For definitions of Sobolev spaces on $\partial\Omega$, we refer to \cite[Chapter I.3]{SohrNSEbook}. We define $C_{0,\sigma}^2(\bar{\Omega}):=\{f\in C^2(\bar{\Omega};\mathbb{R}^3): \textrm{div}\, f=0,\, f|_{\partial\Omega}=0\}.$

As is usual, the notation $C$ denotes a numerical constant possibly depending on parameters that we do not track. This constant may change from line to line. When a constant depends on parameters that we track, we denote this by $C_{\alpha,\beta,\ldots}$.

\section{Preliminaries}
\label{sec.prelim}

Let $\Omega\subset\R^3$ be a smooth bounded domain.\footnote{Below, we apply the results for the fixed smooth domain $\Omega=B_{r_0}$. Hence, we pay no attention here to how the constants in the estimates will depend quantitatively on the regularity of the domain. This dependence is not tracked in \cite{FKS11}. See also Footnote \ref{foot-smoothbdary}.}
Let $-\infty<T_1<T_2<\infty$. Let $a\in L^1(\partial\Omega\times (T_1,T_2);\R^3)$, $b\in L^1(\Omega;\R^3)$ and $F\in L^1(\Omega\times(T_1,T_2);\R^{3\times 3})$ satisfying the compatibility conditions \eqref{e.struc}. 

We consider the following Stokes problem:
\begin{align}\label{e.stovws}
\begin{split}
&\partial_t U-\Delta U+\nabla P=\nabla\cdot F,\quad\nabla\cdot U=0\qquad \mbox{in}\quad \Omega\times(T_1,T_2),\\
&\bar U=a\qquad \mbox{on}\quad\partial \Omega\times(T_1,T_2),\\
&\bar U(\cdot,T_1)=b.
\end{split}
\end{align}
This problem includes in particular the standard Navier-Stokes problem taking $F=-U\otimes U$.

\begin{definition}[very weak solution]
\label{def.vws}
For $a,\, b$ and $F$ given as above,  
we say that $U\in L^1(\Omega\times (T_1,T_2))$ is a very weak solution of \eqref{e.stovws} if for all $\Phi\in C^{1}_{0}([T_1,T_2); C_{0,\sigma}^2(\bar{\Omega})$ and for all $q\in C^\infty(\overline\Omega\times[T_1,T_2];\R)$, we have 
\begin{align}\label{e.vwf}
\begin{split}
&-\int\limits_{T_1}^{T_2}\int\limits_{\Omega}U\cdot(\partial_t\Phi+\Delta\Phi+\nabla q)=-\int\limits_{T_1}^{T_2}\int\limits_{\Omega}F:\nabla\Phi\\
&\qquad\qquad+\int\limits_\Omega b\cdot\Phi(\cdot,T_1)-\int\limits_{T_1}^{T_2}\int\limits_{\partial\Omega}a\cdot(\nabla\Phi\cdot n)-\int\limits_{T_1}^{T_2}\int\limits_{\partial\Omega}(a\cdot n)q.
\end{split}
\end{align}
\end{definition}

We state here an existence and uniqueness result of very weak solutions to the Stokes system. Results in this direction were established in \cite{farwiggaldisohr}, though the version we use below is from \cite{FKS11}.

\begin{theorem}[very weak solutions for the Stokes problem, {\cite[Lemma 1.2]{FKS11}}]\label{theo.flrex}
Let $4\leq s, q<\infty$ such that $\frac2s+\frac3q=1$, and let $r:=\frac23q$. Let $a\in L^s(T_1,T_2;L^r(\Omega))$ and $F=b=0$.\\
Assume that $a$ satisfies the compatibility condition \eqref{e.struc}.\\ 

Then, there exists a unique very weak solution $U\in L^s(T_1,T_2;L^q(\Omega))$ to the Stokes problem \eqref{e.stovws} in the sense of Definition \ref{def.vws}. Morever,
\begin{equation}\label{e.estUab}
\|U\|_{L^s(T_1,T_2;L^q(\Omega))}\leq C_{\Omega,T_1,T_2,q}\|a\|_{L^s(T_1,T_2;L^r(\Omega))}.
\end{equation}
\end{theorem}

Notice that our boundary data $a$ is slightly less general (but general enough for our purposes) than the data considered in \cite[Lemma 1.2]{FKS11}. It follows that any term in the very weak formulation of the equation \eqref{e.vwf} makes sense in a classical integral sense and duality brackets are not needed. The statement of Theorem \ref{theo.flrex} follows directly from \cite[Lemma 1.2]{FKS11}. Indeed, $W^{-\frac1q,q}(\partial\Omega)$ is the dual space of $W^{1-\frac1{q'},q'}(\partial\Omega)$ with $\frac1q+\frac1{q'}=1$, and by Sobolev embedding \cite[Theorem 4.1.3]{Droniou-sobolev}, $W^{1-\frac1{q'},q'}(\partial\Omega)$ embeds into $L^{r'}(\partial\Omega)$ for $r=\frac23q$ and $\frac1r+\frac1{r'}=1$. Therefore, $a\in L^s(T_1,T_2;L^r(\partial\Omega))$ embeds into $L^s(T_1,T_2;W^{-\frac1q,q}(\partial\Omega))$. 

\begin{remark}[non-zero initial data]
We handle non-zero initial data $b$ as a separate issue. Indeed this point is more classical than the case of rough boundary data which is treated in the result above. For non-zero initial data, we rely on the Stokes semigroup estimates in Lebesgue spaces, see for instance \cite{Giga86}.
\end{remark}

\begin{remark}[on alternative linear results with rough boundary data]
It is also possible to rely on the results of Fabes, Lewis and Rivi\`ere for boundary value problems with $L^p$ data obtained in \cite{Lewis72} for the half-space and in \cite{FLR77-pot,FLR77} for bounded smooth domains. However, due to an additional integrability condition on $a\cdot n$ on $\partial\Omega$ in \cite[Theorem (IV.3.3), page 643]{FLR77}, these results require to work in $L^{4+\ep}$, $\ep>0$, rather than $L^4$. Moreover, notice that there is a typo in the statement of \cite[Theorem (IV.3.1)]{FLR77}. The space for $a\cdot n$ is $L^{\bar q}_tL^{\frac23\bar p}_x$, not $L^{\bar q}_tL^{\frac32\bar p}_x$ as written in \cite[Theorem IV.3.1]{FLR77}. Indeed, to get the estimate of the term involving the normal data, namely $\nabla H$, one relies on Theorem IV.2.3.
\end{remark}

\begin{lemma}[Uniqueness of square integrable very weak solutions]\label{uniqueveryweak}
Let $U\in L^2(T_1,T_2;L^2(\Omega))$ be a very weak solution to the Stokes problem \eqref{e.stovws}, in the sense of Definition \ref{def.vws}, with $F=a=b=0$. Then $U\equiv 0$ on $\Omega\times (T_1,T_2)$. 
\end{lemma}
\begin{proof}
The proof relies on classical duality arguments, which we include for completeness. 
Without loss of generality, let $T_1=-1$ and $T_2=0.$ Let $f\in C^{\infty}_{0}(\Omega\times (-1,0);\mathbb{R}^3)$ be arbitrary. Define $\tilde{f}(x,t)=f(x,-t)$ and let $\tilde{\Phi}$ solve
$$\partial_{t}\tilde{\Phi}-\Delta \tilde{\Phi}+\nabla\tilde{q}=\tilde{f}\quad\textrm{and}\quad\textrm{div}\,\tilde{\Phi}=0\quad\textrm{in}\quad \Omega\times (0,\infty),\quad \tilde{\Phi}|_{\partial{\Omega}}=0,\quad \tilde{\Phi}(\cdot,0)=0. $$
From \cite[Theorem 2.7.3, IV]{SohrNSEbook}, there is a classical solution to the above linear problem satisfying $\tilde{\Phi},\tilde{q}\in C^{\infty}(\overline{\Omega\times (\delta,2)})$ for all $\delta\in (0,2)$. Furthermore, since $\tilde{f}$ is supported in time away from zero we can apply \cite[Lemma 2.4.2, Chapter IV]{SohrNSEbook} to infer that $\tilde{\Phi}$ is supported in time away from zero.

 We define
$$\Phi(x,t)= -\tilde{\Phi}(x,-t)\quad\textrm{and}\quad q(x,t)=\tilde{q}(x,-t).$$
Inserting $\Phi$ and $q$ into \eqref{e.vwf} (where $F$, $a$ and $b$ are all zero) gives

$$\int\limits_{-1}^{0}\int\limits_{\Omega} U\cdot f dxds=0\qquad\forall f\in C^{\infty}_{0}(\Omega\times (-1,0);\mathbb{R}^3). $$
 This implies the desired conclusion.
\end{proof} 
\begin{remark}[Very weak solutions and the energy equality]\label{energyequality} 
Let $U\in L^2(T_1,T_2;L^2(\Omega))$ be a very weak solution to the Stokes problem \eqref{e.stovws} in the sense of Definition \ref{def.vws}, with $a=b=0$ and $F\in L^{2}(\Omega\times(T_1,T_2);\R^{3\times 3})$. Applying \cite[Theorem 2.3.1 and Theorem 2.4.1, IV]{SohrNSEbook}, together with Lemma \ref{uniqueveryweak}, gives that
$U\in C([T_1,T_2]; L^{2}_{\sigma}(\Omega))\cap L^{2}(T_1,T_2; W^{1,2}_{0,\sigma}(\Omega))$ with
$$\|U(\cdot,t)\|_{L^{2}(\Omega)}^{2}+2\int\limits_{T_1}^{t}\int\limits_{\Omega} |\nabla U|^2 dxds=-2\int\limits_{T_1}^{t}\int\limits_{\Omega}F:\nabla U dxds\quad\forall t\in [T_1,T_2]. $$
This can also be used to show that very weak solutions to the Navier-Stokes equations (with $a=0$) satisfy the energy equality. For results in this direction, see \cite{Galdi2019}.
\end{remark}

\section{Proof of the localized weak-strong uniqueness results}
\label{sec.wsu}

\subsection{Proof of Theorem \ref{theom.locWSU}}\ 
\label{subsec.wsu}

\medskip

\noindent{\bf Step 1: existence of a `strong' solution to the Navier-Stokes system\footnote{After completion of this paper, we became aware that in \cite{farwiggaldisohr} the existence of a `strong solution' with small non-zero boundary data is also shown. To make this paper self-contained, we include such arguments in Step 1 below.} in $L^4(0,T;L^6(\Omega))$.} 

\noindent\underline{Step 1-a: linear problem.}\\
Our goal is to lift the boundary data $a$ and the initial data $b$ satisfying \eqref{e.intassab} and \eqref{e.struc} by constructing a very weak solution $\bar U$ to the Stokes problem \eqref{e.stovws} in $\Omega\times (0,T)$ with source $F=0$. By Theorem \ref{theo.flrex} of Farwig, Kozono and Sohr, with $s=r=4$ and $q=6$, there exists a unique very weak solution $U_a\in L^4(0,T;L^6(\Omega))$ in the sense of Definition \ref{def.vws} to the Stokes system with $F=b=0$. Hence we construct $\bar U$ the unique very weak solution in the sense of Definition \ref{def.vws} to the Stokes system with $F=0$ as follows:
\begin{equation*}
\bar U=U_a+e^{-t\mathbf A}b.
\end{equation*}
Thanks to the semigroup estimates of \cite{Giga86}, we have for $t\in(0,\infty)$,
\begin{equation*}
\|e^{-t\mathbf A}b\|_{L^6(\Omega)}\leq C_{\Omega}t^{-\frac18}\|b\|_{L^4(\Omega)}.
\end{equation*}
Hence,
\begin{equation}\label{e.linestFLR}
\|\bar U\|_{L^4(0,T;L^6(\Omega))}\leq K_{\Omega,T}\big(\|a\|_{L^4(\partial\Omega\times (t_0,0))}+\|b\|_{L^4(\Omega)}\big),
\end{equation}
for $K_{\Omega,T}\in (0,\infty)$.

\medskip
\noindent\underline{Step 1-b: existence of a unique strong solution $V$ to the perturbed system.}\\
In this step, we construct a mild solution to the perturbed Navier-Stokes system
\begin{align}\label{e.perNS}
\begin{split}
\partial_t V-\Delta V+\nabla Q&=-\nabla\cdot(V\otimes V)-\nabla\cdot(V\otimes \bar U)-\nabla\cdot(\bar U\otimes V)-\nabla\cdot(\bar U\otimes \bar U),\\
\nabla\cdot V&=0\qquad\qquad\qquad\qquad\qquad\qquad \mbox{in}\quad \Omega\times (0,T),\\
V&=0\qquad \mbox{on}\quad\partial\Omega\times (0,T),\\
V(\cdot,0)&=0\qquad \mbox{on}\quad \Omega.
\end{split}
\end{align}
for a small drift $\bar U$. 
We show that the Duhamel equation associated to \eqref{e.perNS}
\begin{equation}\label{e.duhamel}
V(\cdot,t)=-\int\limits_{0}^te^{-(t-s)\mathbf A}\left(\nabla\cdot(V\otimes V)+\nabla\cdot(V\otimes \bar U)+\nabla\cdot(\bar U\otimes V)+\nabla\cdot(\bar U\otimes \bar U)\right)\, ds
\end{equation}
has a unique fixed point in the critical space $L^{4}(0,T;L^{6}(\Omega))$. 

For $t\in (0,T)$, define
$$B(D,E)(\cdot,t):= -\int\limits_{0}^te^{-(t-s)\mathbf A}\nabla\cdot(D\otimes E)\, ds .$$

Using \cite[Proposition 20]{hiebersaal}, we see that 
\begin{equation*}
\left\|e^{-(t-s)\mathbf A}\nabla\cdot(D\otimes E)\right\|_{L^{6}(\Omega)}\leq \frac{C}{(t-s)^{\frac34}}\|D(\cdot,t)\|_{L^{6}(\Omega)}\|E(\cdot,t)\|_{L^{6}(\Omega)},
\end{equation*}
and hence, by Hardy-Littlewood-Sobolev's theorem \cite[Theorem 7.25]{GiaMarbook}, 
\begin{equation}\label{bilinearest}
\left\|B(D,E)\right\|_{L^{4}(0,T;L^{6}(\Omega))}\leq C_0\|D\|_{L^{4}(0,T;L^{6}(\Omega))}\|E\|_{L^{4}(0,T;L^{6}(\Omega))}.
\end{equation}
Define the linear operator $$L(D)(\cdot,t):=-\int\limits_{0}^te^{-(t-s)\mathbf A}\nabla\cdot(D\otimes \bar U+\bar U\otimes D)\, ds. $$
Then by the same reasoning as above
\begin{equation}\label{linearest}
\left\|L(D)\right\|_{L^{4}(0,T;L^{6}(\Omega))}\leq C_1\|D\|_{L^{4}(0,T;L^{6}(\Omega))}\|\bar U\|_{L^{4}(0,T;L^{6}(\Omega))},
\end{equation}
and for the source term quadratic in $\bar U$,
\begin{equation}\label{sourceest}
\left\|B(\bar{U},\bar{U})\right\|_{L^{4}(0,T;L^{6}(\Omega))}\leq C_0\|\bar U\|_{L^{4}(0,T;L^{6}(\Omega))}^2.
\end{equation}

Let \begin{equation}\label{C2def}
C_{2}:=\min\Big(\frac{1}{C_0}, \frac{1}{C_1}\Big).
\end{equation} 

Using \eqref{bilinearest}-\eqref{sourceest}, we can apply \cite[Lemma 4.1]{GIP}. This gives the existence of a fixed point/strong solution $V\in L^{4}(0,T;L^{6}(\Omega))$ provided that \begin{equation}\label{e.barUsmall}
\|\bar U\|_{L^{4}(0,T;L^{6}(\Omega))}<\frac{C_2}{4}.
\end{equation}
Moreover, $$\|V\|_{L^{4}(0,T;L^{6}(\Omega))}\leq 4C_0\|\bar{U}\|^2_{L^{4}(0,T;L^{6}(\Omega))}.$$

In view of the linear estimate \eqref{e.linestFLR}, this is achieved whenever 
\begin{align*}
\|\bar U\|_{L^{4}(0,T;L^{6}(\Omega))}\leq\ &K_{\Omega,T}\big(\|a\|_{L^4(0,T;L^4(\Omega))}+\|b\|_{L^4(\Omega)}\big)\\
=\ &K_{\Omega,T}\kappa\\
<\ &K_{\Omega,T}\bar\kappa=\frac{C_2}{4},
\end{align*}
i.e.
\begin{equation}\label{e.condbarkappa}
\bar\kappa:=\frac{C_2}{4K_{\Omega,T}}.
\end{equation}
Therefore, there is a strong solution such that
\begin{equation}\label{e.estV}
\|V\|_{L^4(0,T;L^6(\Omega))}<4C_0K_{\Omega,T}^2\kappa^2=C_0 C_2 K_{\Omega, T}\kappa.
\end{equation}
The fact that $V$ is the only strong solution in $L^{4}(0,t; L^{6}(\Omega))$ follows from \eqref{bilinearest} and arguments in \cite[Theorem (3.3)]{FJR}.
\medskip

\noindent{\bf Step 2: weak-strong uniqueness result.}

Let $W:=U-\bar U\in L^{4}(\Omega\times (0,T))$. Let $V\in L^{4}(0,T;L^{6}(\Omega))$  be the strong solution constructed in Step 1-b above. Then $V-W$ is a very weak solution to the Stokes system with $a=b=0$ and $F:=U\otimes U-(V+\bar{U})\otimes (V+\bar{U})\in L^{2}(\Omega\times (0,T))$. Applying Remark \ref{energyequality} gives that $V-W$ has finite energy on $\Omega\times (0,T)$ with zero initial data and satisfies the energy equality
\begin{multline*}
\frac12\int\limits_{\Omega}|V-W|^2(x,t)dx+\int\limits_{0}^t\int\limits_{\Omega}|\nabla(V-W)|^2 dxds\\= \int\limits_{0}^{t}\int\limits_{\Omega}(V-W)\otimes(V+\bar U):\nabla (V-W)dxds\quad\forall t\in [0,T).
\end{multline*}
For $t\in [0,T)$, we define
\begin{align*}
\mathcal E(t):= \sup_{s\in [0,t]} \frac{1}{2} \int\limits_{\Omega}|V-W|^2(x,s) dx + \int\limits_{0}^t\int\limits_{\Omega} |\nabla(V-W)|^2(x,s) dxds.
\end{align*}
Using the above energy equality (combined with H\"{o}lder's inequality, interpolation of Lebesgue spaces, the Sobolev embedding theorem and Young's inequality) yields that for $t\in [0,T)$ and some positive universal constant $C$, 
$$\mathcal{E}(t)\leq C\mathcal{E}(t)\int\limits_{0}^{t}\|(V+\bar{U})(\cdot,s)\|_{L^{6}(\Omega)}^4 ds.$$ 
Using this, together with the fact that $V+\bar{U}$ is in $L^{4}(0,T;L^{6}(\Omega))$, allows us to use an absorbing argument to conclude that $\mathcal E(t)=0$ for all $t\in(0,T)$. This concludes the proof of Theorem \ref{theom.locWSU} and of the structure result mentioned in Remark \ref{rem.struc}. Notice that the quantitative estimate \eqref{e.estUL4L6} directly follows from the linear estimate \eqref{e.linestFLR}, the estimate \eqref{e.estV} for the mild solution and the definition of $\bar\kappa$ in \eqref{e.condbarkappa}.

\subsection{Proof of Theorem \ref{theom.locWSUbis}}

The proof of this result only differs from the proof of Theorem \ref{theom.locWSU} in the treatment of the linear evolution of the initial data $b$. Let us outline the changes, which concern only Step 1-a of Section \ref{subsec.wsu}.

We take $\varphi\in C^\infty_c(B_\frac{15}{16})$ a cut-off function such that $\varphi\equiv 1$ on $B_{\frac78}$, $0\leq\varphi\leq 1$ and $|\nabla\varphi|\leq 32$. Thanks to the Bogovskii operator \cite{Galdibook} on the annulus $B_\frac{15}{16}\setminus B_{\frac78}$, there exists a divergence-free extension $E(b)$ of $\varphi b$ defined on $\R^3$ that is compactly supported in $B_\frac{15}{16}$, $E(b)=b$ on $B_{\frac78}\subset B_{r_0}$ and such that 
\begin{equation*}
\|E(b)\|_{L^4(B_1)}\leq C\|b\|_{L^4(B_1)}.
\end{equation*}
Let $\Gamma$ be the heat kernel on $\R^3$. We note that $\Gamma(\cdot-t_0)\star E(b)$ is divergence-free, which implies 
\begin{equation*}
\int\limits_{\partial B_{r_0}}\big(\Gamma(\cdot-t_0)\star E(b)\big)\cdot n=0.
\end{equation*}
Moreover, it follows from \cite[Lemma IV.3.2]{FLR77} that 
\begin{equation}\label{e.intbdaryextendedb}
\big\|\Gamma(\cdot-t_0)\star E(b)\big\|_{L^4(\partial B_{r_0}\times (t_0,0))}\leq C\big((-t_0)^\frac18+(-t_0)^\frac14\big)\|E(b)\|_{L^4(\R^3)}\leq C\|b\|_{L^4(B_1)},
\end{equation}
and 
\begin{equation}\label{e.intintextendedb}
\big\|\Gamma(\cdot-t_0)\star E(b)\big\|_{L^4(t_0,0;L^6(B_{r_0}))}\leq C(-t_0)^\frac18\|E(b)\|_{L^4(\R^3)}\leq C\|b\|_{L^4(B_1)},
\end{equation}
where $C\in (0,\infty)$ denotes as usual a universal constant. 
We can now construct the linear solution $\bar U$ as follows. Let 
\begin{equation*}
\tilde a=a-\Gamma(\cdot-t_0)\star E(b)|_{\partial B_{r_0}}.
\end{equation*}
By Theorem \ref{theo.flrex} of Farwig, Kozono and Sohr \cite{FKS11}, with $s=r=4$ and $q=6$, there exists a unique very weak solution $U_{\tilde a}\in L^4(t_0,0;L^6(B_{r_0}))$ in the sense of Definition \ref{def.vws} to the Stokes system with $F=b=0$ and boundary data $\tilde a$. Hence we construct $\bar U$ the unique very weak solution in the sense of Definition \ref{def.vws} to the Stokes system with $F=0$ as follows:
\begin{equation*}
\bar U=U_{\tilde a}+\Gamma(\cdot-t_0)\star E(b).
\end{equation*}
Hence, the estimates \eqref{e.estUab} and \eqref{e.intintextendedb} lead to
\begin{equation}\label{e.linestFLRbis}
\|\bar U\|_{L^4(t_0,0;L^6(B_{r_0}))}\leq K\big(\|a\|_{L^4(\partial B_{r_0}\times (t_0,0))}+\|b\|_{L^4(B_{r_0})}\big),
\end{equation}
where $K\in (0,\infty)$ denotes a universal constant.

The rest of the proof, Step 1-b to Step 2 of Section \ref{subsec.wsu} are identical, replacing the constant $K_{\Omega,T}$ by $K$. The definition of $\bar\kappa$ becomes
\begin{equation}\label{e.condbarkappabis}
\bar\kappa:=\frac{C_2}{4K}.
\end{equation}

\begin{remark}[on the compatibility conditions]
We emphasize that in the compatibility conditions \eqref{e.struc}, the condition\footnote{We stress that this condition on $a$ comes from the linear result Theorem \ref{theo.flrex} taken from \cite{FKS11}; these conditions are also present in the work \cite{FLR77}.} on the boun\-dary data $a$ is much weaker than the condition on $b$. Indeed we only ask that $a\cdot n$ has mean zero on $\partial\Omega$, while $b\cdot n$ is required to vanish identically on $\partial\Omega$. This fact is the essential redeeming feature that allows the above argument to work. Notice that owing to the fact that $\Gamma(\cdot-t_0)\star E(b)$ is divergence-free, $\big(\Gamma(\cdot-t_0)\star E(b)\big)\cdot n$ has mean zero on $\partial B_{r_0}$ but is not necessarily zero identically.
\end{remark}

\section{Proof of the epsilon regularity result}
\label{sec.proof}

This section is devoted to the proof of Theorem \ref{theom.epreg}. We directly prove the quantitative version, i.e. the bound \eqref{e.concl}. We assume that $U$ is a finite-energy weak solution to the Navier-Stokes equations, i.e. \eqref{e.distribnse} and \eqref{e.en} hold. In addition, we assume that $U$ belongs to $C^\infty(B_1\times(-1,T))$ for all $T\in(-1,0)$, see Footnote \ref{foot.fts}, and satisfies the assumption \eqref{e.HL4+} with $0<\ep\leq\bar\ep$ and $\bar\ep$ defined in \eqref{e.condkappa*}.

\medskip

\noindent{\bf Step 1: finding good space and time scales.}\\
We need to select a space slice and a time slice. The choices are completely independent, so the order in which we select the slices does not matter. Let us first select a space slice $r_0\in(\frac58,\frac78)$. 
By the coarea formula and Fubini's theorem, we have
\begin{equation*}
\int\limits_{-1}^0\int\limits_{B_1}|U|^{4}\, dxdt=\int\limits_{-1}^0\int\limits_0^1\int\limits_{\partial B_r}|U|^{4}\, dS_rdrdt=\int\limits_0^1\int\limits_{-1}^0\int\limits_{\partial B_r}|U|^{4}\, dS_rdtdr.
\end{equation*}
Therefore, the pigeonhole principle implies that there exists $r_0\in(\frac58,\frac78)$ such that
\begin{equation}\label{e.spaceslice}
\int\limits_{-1}^0\int\limits_{\partial B_{r_0}}|U|^{4}\, dS_{r_0}dt\leq 4\int\limits_{-1}^0\int\limits_{B_1}|U|^{4}\, dxdt.
\end{equation}
We now select a time slice. By the pigeonhole principle, there exists $t_0\in(-1,-\frac34)$ such that
\begin{equation}\label{e.timeslice}
\int\limits_{B_1}|U|^{4}(x,t_0)\, dx\leq 4\int\limits_{-1}^0\int\limits_{B_1}|U|^{4}\, dxdt.
\end{equation}
From now on, we call $a:=U|_{\partial B_{r_0}\times (-1,0)}$ and $b:=U(\cdot,t_0)$. By \eqref{e.spaceslice} we have $a\in L^{4}(\partial B_{r_0}\times (-1,0))$ and by \eqref{e.timeslice} we have $b\in L^{4}(B_1)$.  
Moreover, 
\begin{equation*}
\|a\|_{L^4(\partial B_{r_0}\times (t_0,0))}+\|b\|_{L^4(B_{1})}\leq 2\sqrt{2}\|U\|_{L^4(Q_1)}\leq 2\sqrt{2}\ep< 2\sqrt{2}\bar\ep\leq\bar\kappa,
\end{equation*}
with $\bar\kappa$ defined by \eqref{e.condbarkappabis}, on condition that $\bar\ep\leq\frac{\bar\kappa}{2\sqrt{2}}$.

\medskip

\noindent{\bf Step 2: applying weak-strong uniqueness.}\\ 
We now apply the weak-strong uniqueness result of Theorem \ref{theom.locWSUbis} to the solution $U$ on $B_{r_0}\times (t_0,0)$ with the data $a$ and $b$ from Step 1 above. By the quantitative estimate \eqref{e.estUL4L6} it follows that 
\begin{equation}\label{e.critint}
\|U\|_{L^4(t_0,0;L^6(B_{r_0}))}\leq 2\sqrt{2}\ep K(1+C_0C_2).
\end{equation}

\medskip

\noindent{\bf Step 3: conclusion via Lady\v{z}enskaja-Prodi-Serrin.}\\ 
We can now argue in a similar spirit to \cite{sereginnewLPS,BP18}, except we use  the bound \eqref{e.critint} to set up a contraction mapping\footnote{For such a contraction mapping, see for instance \cite[Lemma 12]{BP18}.} related to the localized vorticity equation rather than the velocity.
  In particular, there exists $\bar\eta$ such that for any 
\begin{equation}\label{e.condkappa*}
0<\ep<\frac{\bar\eta}{2\sqrt{2}K(1+C_0C_2)}=:\bar\ep
\end{equation}
we have 
$$\|U\|_{L^{6}(Q_{\frac 12})}\leq M(1+\ep). $$
With this, we can bootstrap in the same way as in \cite{serrin} to obtain the estimate \eqref{e.concl}. This concludes the proof of Theorem \ref{theom.epreg}.

\begin{remark}[on the linear solution]
In this proof, notice that we are not able to assert that the solution $\bar U$ to the linear Stokes problem with rough boundary data $a\in L^4(B_{r_0}\times (t_0,0))$ (see the structure result in Remark \ref{rem.struc}) is smooth in space. However, the linear result is pivotal in order to establish that the original solution $U$ has improved critical integrability, hence is smooth in space. Notice that we obtain $L^{\infty}$ time integrability by appealing to the $L^{\infty}_{t}L^{2}_{x}$ control granted by being finite-energy. We cannot bootstrap further in time without controlling the pressure $p$ up to time 0.
\end{remark}

\begin{remark}[on the half-space]
It remains an open problem as to whether such a proof can be done for establishing epsilon regularity results near a boundary. Indeed, the linear results of Farwig, Kozono and Sohr \cite{FKS11} and of Fabes, Lewis and Rivi\`ere \cite{FLR77} ask for smoothness of the domain $\Omega$, see for instance Theorem \ref{theo.flrex} above. However, we are unable to carry out the slicing procedure of Step~1 above
near a smooth boundary.
\end{remark}

\begin{remark}[comparison with Struwe \cite{struwe1988partial}]
Struwe's approach to the five-dimensional stationary Navier-Stokes equations is technically different; it involves iterating the energy on a sequence of balls, which we avoid. A version of our proof should work in the stationary setting of~\cite{struwe1988partial} as well. There, the smallness condition in $L^4$ could be replaced by smallness in $H^1$, since in five dimensions, the slicing procedure yields boundary data in $H^1(\p B_{r_0}) \subset L^4(\p B_{r_0})$ (four-dimensional Sobolev embedding). The slicing technique has proven useful for the six-dimensional stationary Navier-Stokes equations~\cite{frehse1996existence} and advection-diffusion equation with rough drifts~\cite{albrittondong}, among others.
\end{remark}

\subsection*{Acknowledgement}
DA was supported by NSF Postdoctoral Fellowship Grant No. 2002023. DA is also grateful to ENS Paris for supporting his academic visit to Paris during which this research was initiated. DA thanks Vladim{\'i}r \v{S}ver{\'a}k, who contributed an offhanded comment in 2017 which played a role in the genesis of this paper. CP is partially supported by the Agence Nationale de la Recherche,
project BORDS, grant ANR-16-CE40-0027-01, project SINGFLOWS, grant ANR-
18-CE40-0027-01, project CRISIS, grant ANR-20-CE40-0020-01, by the CY
Initiative of Excellence, project CYNA (CY Nonlinear Analysis) and project CYFI (CYngular Fluids and Interfaces).

\small 
\bibliographystyle{abbrv}
\bibliography{epreg-viaWSU}

\end{document}